\documentclass[11pt,twoside,a4paper,article]{memoir}

\usepackage[latin1]{inputenc}
\usepackage[T1]{fontenc}

\usepackage[english]{babel}

\usepackage{graphicx}

\ifpdf
\fi

\usepackage{microtype}

\usepackage{amsmath}
\usepackage{amssymb}
\usepackage{amsfonts}

\usepackage[amsmath,thmmarks,hyperref]{ntheorem}

\usepackage{mathtools}
\usepackage{empheq}

\usepackage{varioref}

\usepackage{xspace}

\usepackage[justification=RaggedRight]{subfig}

\usepackage{lmodern}
\usepackage[sc]{mathpazo}

\usepackage{url}

\makeatletter
\@ifpackageloaded{mathpazo}{}{%
  \usepackage{fix-cm}
  \usepackage{bbm}}
\makeatother

\ifpdf
\usepackage[pdftitle={},
pdfauthor={Jørgen Ellegaard Andersen, Rasmus Villemoes},
pdfkeywords={mapping class group, abelian moduli space, group cohomology},
colorlinks=true, 
pdftex, backref=none]{hyperref}
\else
\usepackage[pdftitle={The first cohomology of the mapping class group
  with coefficients in algebraic functions on the SL(2, C) moduli
  space},
pdfauthor={Jørgen Ellegaard Andersen, Rasmus Villemoes},
pdfkeywords={mapping class group, multicurve, group cohomology},
colorlinks=true, 
ps2pdf,
bookmarksnumbered=true,
backref=none]{hyperref}
\fi
\usepackage{memhfixc}

\makeatletter
\@ifpackageloaded{mathpazo}{
  \newcommand{\mathsetfont}{\mathbb}
}{
  \newcommand{\mathsetfont}{\mathbbm}}
\makeatother

\newcommand{\DeclareMathSet}[1]{%
  \expandafter\newcommand\csname set#1\endcsname{\mathsetfont{#1}}}

\DeclareMathSet{C} 
\DeclareMathSet{R} 
\DeclareMathSet{Q} 
\DeclareMathSet{Z} 
\DeclareMathSet{N} 

\DeclareMathOperator{\Hom}{Hom}
\DeclareMathOperator{\Aut}{Aut}
\DeclareMathOperator{\Out}{Out}

\newcommand{\inv}{^{-1}}

\DeclarePairedDelimiter{\abs}{\lvert}{\rvert}
\DeclarePairedDelimiter{\inner}{\langle}{\rangle}

\renewcommand{\phi}{\varphi}
\renewcommand{\epsilon}{\varepsilon}
\renewcommand{\rho}{\varrho}

\newcommand{\U}{\mathrm{U}}
\newcommand{\Uone}{\ensuremath{\U(1)}}
\newcommand{\SL}{\mathrm{SL}}
\newcommand{\Sp}{\mathrm{Sp}}
\newcommand{\SU}{\mathrm{SU}}

\newcommand{\ltwo}{\ell^2}
\newcommand{\propT}{Property~(T)\xspace}
\newcommand{\propFH}{Property~(FH)\xspace}

\theoremstyle{plain}
\theoremsymbol{}
\theorembodyfont{\itshape}
\theoremheaderfont{\normalfont\bfseries}
\theoremseparator{.}
\newtheorem{lemma}{Lemma}[section]
\newtheorem{theorem}[lemma]{Theorem}
\newtheorem{proposition}[lemma]{Proposition}
\newtheorem{corollary}[lemma]{Corollary}
\theoremstyle{nonumberplain}
\newtheorem{maintheorem}{Theorem}

\theoremstyle{nonumberplain}
\theorembodyfont{\normalfont}
\theoremheaderfont{\itshape}
\theoremsymbol{\ensuremath{\square}}
\newtheorem{proof}{Proof}

\qedsymbol{\ensuremath{\square}}

\setsecnumdepth{subsection}

\makepagestyle{rv}
\makeoddhead {rv}{}{\textsc{Mapping class groups and the abelian
    moduli space}}{\thepage}
\makeevenhead{rv}{\thepage}{\textsc{J.E. Andersen, R. Villemoes}}{}
\author{Jørgen Ellegaard Andersen \and Rasmus Villemoes}

\title{Cohomology of mapping class groups and the abelian moduli
  space}

\strictpagechecktrue

\begin{document}

\maketitle

\begin{abstract}
  \noindent
  We consider a surface $\Sigma$ of genus $g \geq 3$, either closed or
  with exactly one puncture. The mapping class group $\Gamma$ of
  $\Sigma$ acts symplectically on the abelian moduli space $M =
  \Hom(\pi_1(\Sigma), \Uone) = \Hom(H_1(\Sigma), \Uone)$, and hence
  both $L^2(M)$ and $C^\infty(M)$ are modules over $\Gamma$. In this
  paper, we prove that both the cohomology groups $H^1(\Gamma,
  L^2(M))$ and $H^1(\Gamma, C^\infty(M))$ vanish.
\end{abstract}

\section{Introduction}
\label{sec:introduction}

There are very natural ways to generate infinite dimensional unitary
representations of the mapping class group via representation
varieties of compact Lie groups. Let us here briefly recall the
construction. Let $\Sigma$ be a compact surface of genus $g$, which is
either closed or with one boundary component. Let $G$ be a compact
Lie group and consider the moduli space $M$ of flat
$G$-connections on $\Sigma$, ie.
\begin{equation*}
  M = \Hom(\pi_1(\Sigma), G)/G.
\end{equation*}
If we choose a set of generators for the fundamental group, we get an
induced identification
\begin{equation}
  \label{eq:5}
  M \cong G^{2g}/G
\end{equation}
if $\Sigma$ has a boundary component and
\begin{equation}
  \label{eq:2}
  M \cong \bigl\{(A_i,B_i)\in G^{2g}\mid \prod_{i=1}^g[A_i,B_i] = 1\bigr\}/G  
\end{equation}
if $\Sigma$ is closed.

We use this presentation of this space to define the space of smooth
functions $C^\infty(M)$ on $M$. It is easy to see that this is
independent of the choice of generators. The mapping class group
$\Gamma$ clearly acts on $M$, and this way $C^\infty(M)$ becomes a
module over $\Gamma$.  In the case where $\Sigma$ has one boundary
component, we also observe that both $\Aut(F_{2g})$ and $\Out(F_{2g})$
acts on $M$, where $F_{2g}$ denotes the free group on $2g$ generators.

The biinvariant Haar measure on $G$ induces a measure on $M$
via~\eqref{eq:5} in case $\Sigma$ has one boundary component. In case
$G$ is closed, Goldman~\cite{MR762512} has constructed a symplectic
form $\omega$ on $M$, which induces the Liouville measure
$\omega^n/n!$. In both cases, the mapping class group action preserves
the measure on $M$, so $L^2(M)$, the space of complex-valued, square
integrable functions on $M$, becomes an infinite-dimensional unitary
representation of $\Gamma$.

By work of Goldman~\cite{MR2346275,MR1491446},
Gelander~\cite{MR2448015}, the action of $\Aut(F_n)$ on $\Hom(F_n, G)$
and the action of $\Out(F_n)$ on $\Hom(F_n,G)/G$ are both ergodic for
$n\geq 3$. Furthermore, Pickrell and Xia~\cite{MR1915045,MR2015257},
based on Goldman's results, showed that the action of $\Gamma$ on $M$
is ergodic when $\Sigma$ is closed. When $\Sigma$ has boundary, the
mapping class group preserves the subsets of $M$ defined by requiring
a representation $\rho\colon \pi_1(\Sigma)\to G$ to map each boundary
component into a prescribed conjugacy class in $G$; the action of
$\Gamma$ on each such subset is ergodic.

Ergodicity in particular means that the only invariant functions are
the constants. Hence, letting $L^2_0$ denote the subspace of $L^2$
corresponding to functions with mean value $0$, the above results may
be interpreted as the vanishing of certain $0$'th cohomology groups,
such as $H^0(\Aut(F_n), L^2_0(G^n))$ and $H^0(\Gamma, L^2_0(M))$.

It is very natural to ask if $H^1(\Gamma, L^2(M))$ vanishes both in
case where $\Sigma$ is closed and in the case where $\Sigma$ has one
boundary component. In the latter case, we can also ask if
$H^1(\Aut(F_{2g}), L^2(M))$ and $H^1(\Out(F_{2g}), L^2(M))$ vanishes.
As it is well known, answering any of these questions in the negative
implies that the corresponding group does not have Kazhdan's
property~(T) \cite{MR2415834}.  In case $\Sigma$ is closed, Andersen
has in~\cite{math.QA/0706.2184} established that the mapping class
group does not have Kazhdan's property~(T) by using the TQFT quantum
representations of $\Gamma$. We, however, do not expect that any of
these cohomology groups are non-vanishing and so will not shed light
on this question.

In this paper we answer the first question affirmatively in the
abelian case, where $G = \Uone$ (see Theorem~\ref{thm:1}).
\begin{maintheorem}
  For $G = \Uone$ we have that $H^1(\Gamma, L^2(M)) = 0$.
\end{maintheorem}

The proof of this theorem uses the fact that for $g\geq 2$, the group
$\Sp(2g,\setZ)$ is known to have property~(T), the Hochschild-Serre
exact sequence, along with the following result (Theorem~\ref{thm:4})
which holds for all unitary representations:
\begin{maintheorem}
  Let $\Gamma\to \U(V)$ be a unitary represention of the mapping class
  group on a real or complex Hilbert space $V$. For a Dehn twist
  $\tau_\gamma$, let $V_\gamma$ denote the subspace of $V$ fixed under
  $\tau_\gamma$, and let $p_\gamma\colon V\to V_\gamma$ denote the
  orthogonal projection. Then $p_\gamma u(\tau_\gamma) = 0$ for any
  cocycle $u\colon \Gamma\to V$ and any simple closed curve $\gamma$.
\end{maintheorem}

We are also able to prove the analogous result of Theorem~\ref{thm:1}
when we replace $L^2$-functions by smooth functions
(Theorem~\ref{thm:2}).
\begin{maintheorem}
  For $G = \Uone$ we have that $H^1(\Gamma, C^\infty(M)) = 0$.
\end{maintheorem}

These two results should be compared to the main results
from~\cite{0710.2203} and~\cite{0802.4372}. In the latter, we
considered the case $G = \SL_2(\setC)$ and the space $\mathcal{O} =
\mathcal{O}(\mathcal{M}_{\SL_2(\setC)})$ of regular functions on the
moduli space (this makes sense, since~\eqref{eq:5} and~\eqref{eq:2}
give the moduli space the structure of an algebraic variety). The
conclusion in that case was that $H^1(\Gamma, \mathcal{O}) = 0$. In
the former paper, on which the latter is based, we considered the
algebraic dual module, $\mathcal{O}^* = \Hom(\mathcal{O},\setC)$, and
the conclusion in that case was that $H^1(\Gamma, \mathcal{O}^*)$ can
be written as a countable direct product of finite-dimensional
components, of which at least one is non-zero.

This paper is organized as follows. In the next section, we briefly
describe our motivation for studying this problem, apart from its
connection to Property~(T). In section~\ref{sec:twists-relations}, we
briefly recall certain well-known facts about mapping class groups:
relations between Dehn twists, the action of a twist on a homology
element, and generation of the Torelli group by bounding pair maps.
The main purpose of section~\ref{sec:unit-repr} is to prove that for
$g\geq 3$, a certain necessary condition for the vanishing of the
cohomology group $H^1(\Gamma, V)$ is always satisfied, for any unitary
representation $V$ of $\Gamma$ (this is the above-mentioned
Theorem~\ref{thm:4}). We also quote the results about $\Sp(2g,\setZ)$
and property~(T) which we need. Section~\ref{sec:funct-abel-moduli} is
devoted to describing an orthonormal basis for the space of
$L^2$-functions on the abelian moduli space. This basis has two nice
properties: The mapping class group acts by permuting basis elements,
and there is a simple condition for determining if an $L^2$-function
is smooth in terms of its coefficients in this basis. Finally, in
section~\ref{sec:cohom-comp}, we prove the two main theorems quoted
above.

\section{Motivation}
\label{sec:motivation}

The motivation for studying the first cohomology group of the mapping
class group with coefficients in a space of functions on the moduli
space came from \cite{0611126}. In that paper, the first author
studied deformation quantizations, or star products, of the Poisson
algebra of smooth functions on the moduli space $M_G$ for $G=\SU(n)$.
The construction uses Toeplitz operator techniques and produces a
family of star products parametrized by Teichmüller space. In
\cite{0611126} the problem of turning this family into one mapping
class group invariant star product was reduced to a question about the
first cohomology group of the mapping class group with various twisted
coefficients.  Specifically, one of the results in \cite{0611126}
(Proposition~6) is that, provided the cohomology group $H^1(\Gamma,
C^\infty(M_G))$ vanishes, one may find a $\Gamma$-invariant
equivalence between any two equivalent star products.  Since it is
easy to see that the only $\Gamma$-invariant equivalences are the
multiples of the identity, this immediately implies that within each
equivalence class of star products, there is at most one
$\Gamma$-invariant star product.

Considering the results of \cite{MR2135450}, \cite{0611126} and the
present paper, we get the following application.
\begin{theorem}
  \label{thm:5}
  For $G = \Uone$, there is a unique mapping class group invariant
  star product on $M_G$.
\end{theorem}
Existence follows from \cite{MR2135450} and uniqueness from
\cite{0611126} and Theorem~\ref{thm:2} below. 

\section{Group cohomology}
\label{sec:group-cohomology}

In this section we will introduce some terminology and basic results
which will be used throughout the rest of the paper. Let $G$ be a
group. A $G$-module is a module over the integral group ring
$\setZ G$, or equivalently, an abelian group $M$ together with a
homomorphism $\pi\colon G\to\Aut(M)$.

A cocycle on $G$ with values in $M$ is a map $u\colon G\to
M$ satisfying the \emph{cocycle condition}
\begin{equation}
  \label{eq:12}
  u(gh) = u(g) + g u(h)
\end{equation}
for all $g,h\in G$. Here, and elsewhere, we suppress the homomorphism
defining the action from the notation; the last term in \eqref{eq:12}
should be read $\pi(g) u(h)$. The space of all cocycles is denoted $Z^1(G,M)$.
It is easy to see from \eqref{eq:12} that a cocycle is determined by
its values on a set of generators of $G$. If $1\in G$ denotes the
identity element, it is easy to see that $u(1) = 0$. From this it
follows that $u(g\inv) = -g\inv u(g)$ for any $g\in G$. It is also
easy to deduce the formula $u(ghg\inv) = (1-ghg\inv)u(g) + g u(h)$.
These observations will be used without further comment.

A cocycle is said to be a coboundary if it is of the
form $g\mapsto v - gv = (1-g)v$ for some $v\in M$. The space
of coboundaries is denoted $B^1(G, M)$, and the first cohomology
group of $G$ with coefficients in $M$ is the quotient
\begin{equation}
  \label{eq:13}
  H^1(G, M) = Z^1(G,M) / B^1(G, M).
\end{equation}
Notice that in the special case where the group acts trivially on $M$,
the cocycle condition simply means that $u$ is a homomorphism, and the
space of coboundaries vanish. Hence, in that case we have $H^1(G, M) =
\Hom(G, M)$. If $\pi\colon G\to \Aut(M)$ is the homomorphism defining
the action, we may also denote $H^1(G, M)$ by $H^1(G, \pi)$.

If $G$ is a topological group, $M$ is a topological vector space and
the action of $G$ on $M$ is continuous, one may equip $Z^1(G, M)$ with
the topology of uniform convergence over compact subsets. In this
topology, $B^1(G, M)$ may or may not be closed in $Z^1(G, M)$; in any
case, the quotient
\begin{equation}
  \label{eq:4}
  \overline{H^1}(G, M) = Z^1(G, M) / \overline{B^1(G, M)}
\end{equation}
is known as the reduced cohomology of $G$ with coefficients in $M$.

\begin{proposition}
  \label{prop:1}
  Assume $1\to K\to G\to Q\to 1$ is a short exact sequence of groups,
  and that $M$ is a $G$-module on which $K$ acts trivially (hence
  making $M$ a $Q$-module). Then there is an exact sequence
  \begin{equation}
    \label{eq:1}
    0\to H^1(Q, M) \to H^1(G, M) \to H^1(K, M)^G.
  \end{equation}
\end{proposition}
Here $H^1(K, M)^G$ denotes the subset of $H^1(K, M) = \Hom(K, M)$
which is invariant under $G$. The action is given by $(g\cdot u)(k) =
g\inv u(gkg\inv)$, so an invariant homomorphism is one that satisfies
the equivariance condition
\begin{align}
  \label{eq:30}
  u(gkg\inv) = gu(k)
\end{align}
for all $g\in G$ and $k\in K$.

This exact sequence comes from an abstract beast known as the
Hochschild-Serre spectral sequence, and really continues with two
$H^2$ terms. Also, in the more general case one does not need to
require that $K$ acts trivially on $M$; instead, the cohomology of $Q$
is taken with coefficients in the submodule $M^K$ invariant under $K$.
However, we only need the part of the exact sequence shown in
\eqref{eq:1}, and we are able to give an explicit hands-on proof of
this proposition which does not involve a spectral sequence.
\begin{proof}
  The first map above is given by precomposing a cocycle $u\colon Q\to
  M$ with the projection map $\pi\colon G\to Q$. This clearly maps
  cocycles to cocycles. If $u\in Z^1(Q, M)$ is the coboundary of some
  element $v\in M$, then the cocycle $u\circ \pi\in Z^1(G, M)$ is also
  the coboundary of $v$. Hence the first map above is
  well-defined. Furthermore, if $u\circ\pi$ is a coboundary, then
  $u(q) = u(\pi(\widetilde q)) = (1-\widetilde q) v = (1-q)v$, where
  $\widetilde q$ is any element of $G$ mapping to $q$ under
  $\pi$. This proves that the first map above is injective, and hence
  proves exactness at $H^1(Q, M)$.

  The second map above is given by restricting a cocycle $u\colon G\to
  M$ to $K$. It is easy to see that the restricted map is a
  homomorphism from $K$, and that restricting a coboundary gives the
  zero map, so that the map is well-defined. To see that the map
  actually takes values in the space of invariant homomorphisms
  follows from the little calculation
  \begin{align*}
    (g\cdot u)(k) &= g\inv u(gkg\inv)\\
    &= g\inv \bigl( (1-gkg\inv)u(g) +
    gu(k) \bigr)\\
    &= u(k) + (1-k) g\inv u(g)\\
    &= u(k)
  \end{align*}
  since $k$ acts trivially on $M$.

  Clearly, if $u$ is a cocycle $Q\to M$, the composition $K\to G\to
  Q\to M$ is zero, so the image of the first map is contained in the
  kernel of second. Conversely, assume that $u\colon G\to M$ is a
  cocycle which satisfies $u(k) = 0$ for any $k\in K$. For any $q\in
  Q$, choose some $g\in G$ mapping to $q$, and put $\widetilde u(q) =
  u(g)$. This is well-defined, as another choice $g'$ of lift would
  differ from $g$ by an element $k\in K$, and then $u(g') = u(gk) =
  u(g) + gu(k) = u(g)$. If $q_1, q_2\in Q$, choose lifts $g_1, g_2\in
  G$. Then the product $g_1g_2$ is a lift of $q_1q_2$, and we have
  \begin{equation*}
    \widetilde u(q_1q_2) = u(g_1g_2) = u(g_1) + g_1u(g_2) = \widetilde
    u(q_1) + q_1 \widetilde u(q_2),
  \end{equation*}
  so $\widetilde u$ is a cocycle on $Q$. This proves exactness at
  $H^1(G, M)$.
\end{proof}

\section{Twists and relations}
\label{sec:twists-relations}

\begin{lemma}
  \label{lem:5}
  Dehn twists on disjoint curves commute.
\end{lemma}
\begin{lemma}
  \label{lem:6}
  If $\alpha$ and $\beta$ are simple closed curves intersecting
  transversely in a single point, the associated Dehn twists are
  \emph{braided}. That is, $\tau_\alpha\tau_\beta\tau_\alpha =
  \tau_\beta\tau_\alpha\tau_\beta$.
\end{lemma}
\begin{lemma}
  \label{lem:4}
  If $\alpha$ is a simple closed curve on $\Sigma$ and $f\in \Gamma$,
  we have $f\circ \tau_\alpha\circ f\inv = \tau_{f(\alpha)}$.
\end{lemma}
\begin{lemma}[Chain relation]
  \label{lem:7}
  Let $\alpha$, $\beta$ and $\gamma$ be simple closed curves in a
  two-holed torus as in Figure~\ref{fig:chain-relation}, and let
  $\delta$, $\epsilon$ denote curves parallel to the boundary
  components of the torus. Then $(\tau_\alpha\tau_\beta\tau_\gamma)^4
  = \tau_\delta\tau_\epsilon$.
\end{lemma}
\begin{figure}[hbt]
  \centering
  \quad
  \subfloat[A two-holed torus.\label{fig:chain-relation-real}]
  {\includegraphics[trim=-10 0 -10 0,clip]{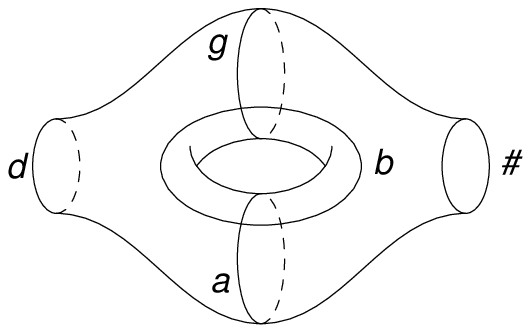}}
  \hfill
  \subfloat[A more schematic picture.\label{fig:chain-relation-schematic}]
  {\includegraphics[trim=-10 0 -10
    0,clip]{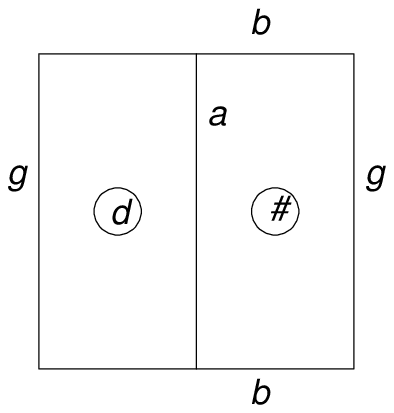}}
  \quad\strut
  \caption{The chain relation.}
  \label{fig:chain-relation}
\end{figure}
\begin{lemma}[Lantern relation]
  \label{lem:8}
  Consider the surface $\Sigma_{0,4}$, ie. a sphere with four holes.
  Let $\gamma_i$ denote the $i$'th boundary component, $0\leq i\leq
  3$, and $\gamma_{ij}$ a loop enclosing the $i$'th and $j$'th
  boundary components, $1\leq i<j\leq 3$. Let $\tau_i =
  \tau_{\gamma_i}$ and $\tau_{ij} = \tau_{\gamma_{ij}}$. Then
    \begin{align}
      \label{eq:62}
      \tau_0\tau_1\tau_2\tau_3 = \tau_{12}\tau_{13}\tau_{23}.
    \end{align}
\end{lemma}
For a picture of the lantern relation, see the left-hand part of
Figure~\vref{fig:lantern-non-sep}.
\begin{corollary}
  \label{cl:4}
  If $g\geq 2$, the Dehn twist on a boundary component of $\Sigma_{g,r}$
  can be written in terms of Dehn twists on non-separating curves.
\end{corollary}
\begin{proof}
  The assumption on the genus implies that we may find an embedding of
  $\Sigma_{0,4}\to\Sigma_{g, r}$ such that $\gamma_0$ is mapped to the
  boundary component in question and the remaining six curves involved
  in the lantern relation are mapped to non-separating curves (think
  of $\Sigma_{g,r}$ as being obtained by gluing three boundary
  components of $\Sigma_{g-2, r+2}$ to $\gamma_1$, $\gamma_2$ and
  $\gamma_3$, respectively). Then the relation $\tau_0 =
  \tau_{12}\tau_{13}\tau_{23}\tau_3\inv\tau_2\inv\tau_1\inv$ also
  holds in $\Gamma_{g,r}$.
\end{proof}
\begin{corollary}
  \label{cl:5}
  When $g\geq 3$, $\Gamma_{g,r}$ is generated by Dehn twists on
  \emph{non-separating} curves.
\end{corollary}
\begin{proof}
  We already know that the mapping class group is generated by Dehn
  twists. If $\gamma$ is a separating curve in $\Sigma$, cut $\Sigma$
  along $\gamma$ and apply Corollary~\ref{cl:4} to the component which
  has genus $\geq 2$, showing that $\tau_\gamma$ can be written in
  terms of twists on non-separating curves in $\Sigma$.
\end{proof}

\subsection{Action on homology}
\label{sec:action-homology}

Let $\gamma$ be a simple closed curve on $\Sigma$, let
$\vec\gamma$ denote one of its oriented versions, and let
$[\vec\gamma]\in H_1(\Sigma)$ denote the homology class of
$\vec\gamma$. Then for any homology class $m\in H_1(\Sigma)$, the
action of $\tau_\gamma$ on $m$ is given by the formula
\begin{align}
  \label{eq:8}
  \tau_\gamma m = m + i([\vec\gamma], m) [\vec\gamma],
\end{align}
where $i(\cdot, \cdot)$ denotes the intersection pairing on homology.
Clearly, the right-hand side of \eqref{eq:8} is independent of the
choice of orientation of $\gamma$. By induction and using linearity
and antisymmetry of $i$, \eqref{eq:8} may be generalized to
\begin{align}
  \label{eq:9}
  \tau_\gamma^n m = m + ni([\vec\gamma], m) [\vec\gamma]
\end{align}
This formula immediately implies an important fact.
\begin{lemma}
  \label{lem:11}
  If $\tau_\gamma$ acts non-trivially on $m$, the orbit
  $\{\tau_\gamma^n m \mid n\in\setZ\}$ is infinite.
\end{lemma}

Let $(x_1, y_1, \ldots, x_g, y_g)$ be a $2g$-tuple of oriented simple
closed curves representing a \emph{symplectic basis} for
$H_1(\Sigma)$; that is, $i(x_j, y_j) = 1$ and $i(x_j,x_k) = i(y_j,y_k) =
0$ for all $j,k$ and $i(x_j,y_k) = 0$ for $j\not=k$. Such a basis
induces a norm on $H_1(\Sigma)$ by putting
\begin{equation}
  \label{eq:10}
  \abs{m} = \abs{a_1} + \abs{b_1} + \cdots + \abs{a_g} + \abs{b_g}
\end{equation}
for $m = a_1x_1 + b_1y_1 + \cdots + a_gx_g + b_gy_g$.

We will need the following little technical result later.
\begin{lemma}
  \label{lem:12}
  Given any symplectic basis and any non-zero homology element $m$,
  there exists a curve $\gamma$ such that at least one of the
  sequences $\abs{\tau_\gamma^n m}$, $\abs{\tau_\gamma^{-n}}$,
  $n=0,1,2,\ldots$, is strictly increasing.
\end{lemma}
\begin{proof}
  Let $(a_1, b_1, \ldots, a_g, b_g)\in\setZ^{2g}$ be the coordinates
  of $m$ with respect to the given basis. At least one of these
  coordinates is non-zero. Assume WLOG $a_1 \not= 0$ and put $\gamma =
  b_1$. Then, for any $n\in\setZ$, the coordinates of $\tau_\gamma^n
  m$ are
  \begin{equation*}
    (a_1, b_1 + na_1, a_2, b_2, \ldots, a_g, b_g)
  \end{equation*}
  by \eqref{eq:9} above. Then clearly if $a_1$ and $b_1$ have the same
  sign ($b_1$ may be $0$), the sequence $\abs{\tau_\gamma^nm}$ is
  increasing, while if $a_1$ and $b_1$ have opposite signs the
  sequence $\abs{\tau_\gamma^{-n}m}$ is increasing.
\end{proof}
Note that we may in fact in all cases choose the Dehn twist from a
finite collection of twists.

\subsection{The Torelli group}
\label{sec:torelli-group}

An important subgroup of $\Gamma$ is the \emph{Torelli group}
$\mathcal{T}$, which by definition is the kernel of the homomorphism
$\Gamma\to\Sp(H_1(\Sigma)) \cong \Sp_{2g}(\setZ)$. By work of
Johnson~\cite{MR529227}, it is known that the Torelli group is
generated by \emph{genus $1$ bounding pair maps}. By definition, a
bounding pair is a pair $(\gamma, \delta)$ of non-isotopic,
non-separating simple closed curves $\gamma, \delta$, such that the
union $\gamma\cup \delta$ separates the surface. The genus of such a
pair is, in the case of a closed surface, the minimum of the genera of
the two subsurfaces separated by $\gamma\cup \delta$, and in the case
of a once-puncture surface, the genus of the subsurface not containing
the puncture. The bounding pair map (or BP map) associated to
$(\gamma,\delta)$ is the map $\tau_\gamma\tau_\delta\inv$. Since
$\gamma$ and $\delta$ are homologous, $\tau_\gamma$ and $\tau_\delta$
acts identically on the homology of $\Sigma$, so it is trivial that
bounding pair maps belong to the Torelli group.

%

\section{Unitary representations}
\label{sec:unit-repr}

In this section, we will observe some general facts about cocycles on
the mapping class group with values in a unitary representation.
Throughout this section, let $V$ be a real or complex Hilbert
space endowed with an action of $\Gamma$ preserving the inner product.

For a simple closed curve $\gamma$, we let $V_\gamma =
V^{\tau_\gamma}$ denote the set of vectors fixed under the action of
the twist $\tau_\gamma$, and we let $p_\gamma\colon V\to V_\gamma$
denote the orthogonal projection onto the (obviously closed) subspace
$V_\gamma$. If $\alpha$ and $\gamma$ are disjoint simple closed
curves, the unitary actions $\tau_\alpha$ and $\tau_\gamma$ on $V$
commute. Hence the associated projections $p_\alpha$ and $p_\gamma$
commute with each other and with $\tau_\alpha,\tau_\gamma$. If
$\phi\tau_\alpha \phi\inv = \tau_\beta$, then $\phi p_\alpha \phi\inv
= p_\beta$ for $\phi\in\Gamma$.

\subsection{A satisfied coboundary condition}
\label{sec:satisf-cobo-cond}

From now on, let $u$ denote a fixed cocycle. We will now investigate a
certain condition for $u$ to be a coboundary, which will turn out to
be satisfied whenever $g\geq 3$. If $u(\phi) = (1-\phi)v$ for some
vector $v$, it is clear that $u(\phi)$ is killed by the projection
onto the subspace $V^\phi$ fixed by $\phi$. Hence if $\alpha$ is a
simple closed curve, it is natural to consider the entity $p_\alpha
u(\tau_\alpha)$. The main theorem of this section is
\begin{theorem}
  \label{thm:4}
  Let $\Sigma$ be a surface of genus at least $3$ and let $V$ be a
  unitary representation of the mapping class group $\Gamma$ of
  $\Sigma$. For any cocycle $u\colon \Gamma\to V$ and any simple
  closed curve $\alpha$ we have $p_\alpha u(\tau_\alpha) = 0$.
\end{theorem}
The proof of this theorem only requires the simple relations in the
mapping class group mentioned in Section~\ref{sec:twists-relations}.

We will use the shorthand notation $s_\alpha$ for $p_\alpha
u(\tau_\alpha)$.
\begin{lemma}
  \label{lem:9}
  The entity $s$ is natural in the sense that $s_{\phi(\alpha)} = \phi
  s_\alpha$ for $\phi\in\Gamma$ and any simple closed curve $\alpha$.
\end{lemma}
\begin{proof}
  Since $\tau_{\phi(\alpha)} = \phi\tau_\alpha \phi\inv$, it is easy
  to see that $p_{\phi(\alpha)} = \phi p_\alpha \phi\inv$. Hence
  \begin{align*}
    s_{\phi(\alpha)} &= p_{\phi(\alpha)} u(\tau_{\phi(\alpha)})\\
    &= \phi p_\alpha \phi\inv u(\phi\tau_\alpha \phi\inv)\\
    &= \phi p_\alpha \phi\inv \bigl( (1-\phi \tau_\alpha \phi\inv)
    u(\phi) + \phi u(\tau_\alpha) \bigr)\\
    &= \phi p_\alpha u(\tau_\alpha)\\
    &= \phi s_\alpha
  \end{align*}
  as claimed.
\end{proof}
\begin{lemma}
  \label{lem:13}
  Let $\alpha$ be a simple closed curve, and let $\phi\in\Gamma$ be
  any element commuting with $\tau_\alpha$. Then $\phi s_\alpha =
  s_\alpha$.
\end{lemma}
\begin{proof}
  We have $\phi\tau_\alpha = \tau_\alpha \phi$. Applying $u$ and the
  cocycle condition we obtain the equation $u(\phi) + \phi
  u(\tau_\alpha) = u(\tau_\alpha) + \tau_\alpha u(\phi)$. Applying
  $p_\alpha$ on both sides, the terms involving $u(\phi)$ cancel
  (since obviously $p_\alpha \tau_\alpha = p_\alpha$), so we obtain
  $p_\alpha \phi u(\tau_\alpha) = s_\alpha$. The claim then follows
  from the fact that $p_\alpha$ and $\phi$ commute.
\end{proof}

Assume $\alpha$ and $\beta$ are two non-separating simple closed
curves such that $\alpha\cup \beta$ is non-separating, and consider
the number $c_{\alpha\beta} = \inner{s_\alpha, s_\beta}$.
\begin{lemma}
  \label{lem:10}
  The number $c_{\alpha\beta}$ only depends on the cocycle $u$, not on
  the pair $(\alpha, \beta)$ used to compute it.
\end{lemma}
\begin{proof}
  Let $(\alpha', \beta')$ be any other pair such that $\alpha'\cup
  \beta'$ does not separate $\Sigma$. Then, by the classification of
  surfaces, there is a diffemorphism $\phi\in\Gamma$ such that
  $\phi(\alpha) = \alpha'$ and $\phi(\beta) = \beta'$. Then by the
  naturality from Lemma~\ref{lem:9} we have
  \begin{equation*}
    \inner{s_{\alpha'}, s_{\beta'}} = \inner{s_{\phi(\alpha)},
      s_{\phi(\beta)}} = \inner{\phi s_\alpha, \phi s_\beta} =
    \inner{s_\alpha, s_\beta}
  \end{equation*}
  since $\phi$ acts unitarily.
\end{proof}

The vector $s_\alpha = p_\alpha u(\tau_\alpha) \in V$ obviously only
depends on the cohomology class $[u]\in H^1(\Gamma, V)$ of $u$. Hence,
we have essentially proved that there exists a well-defined map
$c\colon H^1(\Gamma, V)\to \setC$, whose value on $[u]$ is given by
picking any two jointly non-separating simple closed curves
$\alpha,\beta$ and computing the number $c([u]) = \inner{p_\alpha
  u(\tau_\alpha), p_\beta u(\tau_\beta)}$.
\begin{lemma}
  \label{lem:14}
  When $g\geq 3$, the map $c$ is identically $0$.
\end{lemma}
\begin{proof}
  In any surface of genus at least $2$, one may embed the two-holed
  torus relation (Lemma~\ref{lem:7}) in such a way that $\gamma$ and
  $\delta$ are non-separating (the curves $\alpha,\beta,\gamma$
  occuring in the two-holed torus relation are always non-separating).
  If the genus of the surface is at least $3$, the complement of the
  two-holed torus is a surface of genus at least $1$. Hence, in that
  subsurface we may find a sixth non-separating curve $\eta$. Observe
  that $\eta$ makes a non-separating pair with each of the other five
  curves. See Figure~\ref{fig:sixcurves}.
  \begin{figure}[htb]
    \centering
    \includegraphics{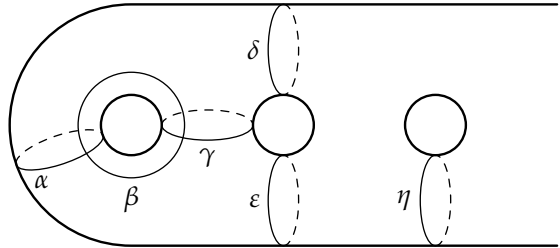}
    \caption{A two-holed torus embedded in a surface of genus $\geq 3$.}
    \label{fig:sixcurves}
  \end{figure}

  Applying $u$ and the cocycle condition repeatedly to the two-holed
  torus relation yields the equation
  \begin{equation}
    \label{eq:17}
    u(\tau_\alpha) + \tau_\alpha u(\tau_\beta) + \cdots =
    u(\tau_\delta) + \tau_\delta u(\tau_\epsilon).
  \end{equation}
  The dots on the left-hand side represent $10$ terms involving
  various actions of $\tau_\alpha,\tau_\beta,\tau_\gamma$ on the
  values of $u$ on these twists. Since each of the five curves is
  disjoint from $\eta$, we have $\tau_\alpha^{\pm1} s_\eta = s_\eta$,
  and similarly for $\beta,\gamma,\delta,\epsilon$. Now we take the
  inner product of \eqref{eq:17} with $s_\eta$ to obtain
  \begin{equation}
    \label{eq:18}
    4\inner{u(\tau_\alpha), s_\eta} + 4\inner{u(\tau_\beta), s_\eta} +
    4\inner{u(\tau_\gamma), s_\eta} = \inner{u(\tau_\delta), s_\eta} +
    \inner{u(\tau_\epsilon), s_\eta}
  \end{equation}
  using the fact that $\inner{\phi x, y} = \inner{x, \phi\inv y}$.
  But since $\tau_\alpha s_\eta = s_\eta$, we also have $p_\alpha
  s_\eta = s_\eta$, and since the projection $p_\alpha$ is
  self-adjoint, the first term in \eqref{eq:18} is equal to
  $4\inner{s_\alpha, s_\eta} = 4c$. Similar remarks apply to the other
  terms, so \eqref{eq:18} reduces to $12c = 2c$, so $c = 0$.
\end{proof}

Now we are ready to prove the main result of this section.
\begin{proof}[Theorem~\ref{thm:4}]
  We first treat the case where $\alpha$ is non-separating.  We cannot
  simply put $\alpha = \beta$ in the computation of $c$, since
  $(\alpha,\alpha)$ is not a non-separating pair. But when the surface
  has genus at least $3$, we may embed the lantern relation
  (Lemma~\ref{lem:8}) in such a way that all seven curves are
  non-separating. Furthermore, it can be done in such a way that
  $\gamma_0$ makes a non-separating pair with each of the other six
  curves. On Figure~\ref{fig:lantern-non-sep} this is shown for a
  genus~$3$ surface; note that the shown surface has been cut along
  $\gamma_0$. The right-hand part of the cut surface (a sphere with
  four holes) could be replaced by a surface with any genus and four
  boundary components. Now the cocycle condition applied to the
  lantern relation gives
  \begin{equation}
    \label{eq:19}
    u(\tau_0) + \tau_0 u(\tau_1\tau_2\tau_3) = u(\tau_{12}\tau_{13}\tau_{23}).
  \end{equation}
  Finally, taking the inner product with $s_{\gamma_0}$ on both sides
  and applying computations similar to those above, we get
  $\inner{s_{\gamma_0}, s_{\gamma_0}} = \inner{u(\tau_0), s_{\gamma_0}}
  = 0$. Hence $s_{\gamma_0} = 0$, and by naturality
  (Lemma~\ref{lem:9}) this holds for any non-separating curve.

  \begin{figure}[htb]
    \centering
    \includegraphics{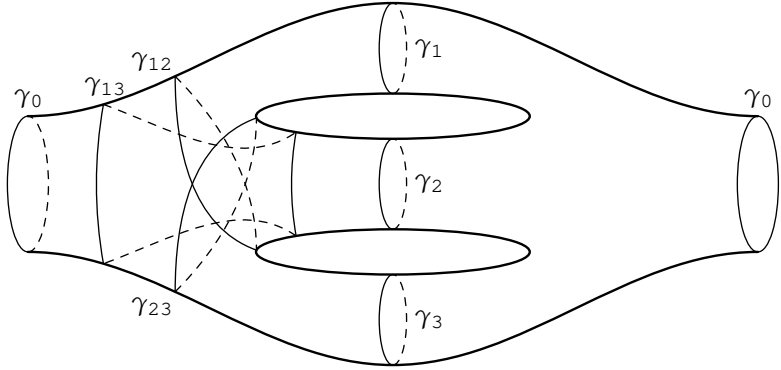}
    \caption{An embedding of the lantern relation such that all seven
      curves are non-separating. The $\gamma_0$ on the left is
      identified with that on the right.}
    \label{fig:lantern-non-sep}
  \end{figure}

  If $\alpha$ is separating, we use the fact that one of the sides of
  $\alpha$ has genus at least $2$ and Corollary~\ref{cl:4} to write
  $\tau_\alpha$ as a product of twists in six non-separating curves.
  For some appropriate choice of signs $\epsilon_j$, we thus have $\tau_\alpha =
  \prod_{j=1}^6 \tau_j^{\epsilon_j}$,
  where the $\tau_j$ are the twists in the appropriate non-separating
  curves disjoint from $\alpha$. Now apply the cocycle condition and
  take the inner product with $s_\alpha$ to obtain
  \begin{equation*}
    \inner{s_\alpha,s_\alpha} = \inner{u(\tau_\alpha), s_\alpha} =
    \inner{u(\tau_1^{\epsilon_1}), s_\alpha} + \cdots +
    \inner{ \tau_1^{\epsilon_1}\tau_2^{\epsilon_2}\tau_3^{\epsilon_3}
      \tau_4^{\epsilon_4}\tau_5^{\epsilon_5} u(\tau_6^{\epsilon_6}), s_\alpha}.
  \end{equation*}
  By Lemma~\ref{lem:13}, $\tau_j^{\pm1}s_\alpha = s_\alpha$, so using
  the unitarity of the action this reduces to
  \begin{equation*}
    \inner{s_\alpha,s_\alpha} = \sum_{j=1}^6 \inner{u(\tau_j^{\epsilon_j}), s_\alpha}
  \end{equation*}
  Finally, we conclude that each term on the right-hand side vanishes
  by writing $s_\alpha$ as $\smash{p_j s_\alpha}$, moving the
  self-adjoint projection $p_j$ to $\smash{u(\tau_j^{\epsilon_j})}$
  and using that $s_\beta = 0$ for non-separating curves $\beta$.
\end{proof}

\subsection{Property (T) and Property (FH)}
\label{sec:property-t}

Two properties of topological groups, known as \propT and \propFH,
respectively, are intimately related to the cohomology of groups with
coefficients in real or complex Hilbert spaces. A thorough exposition
of these properties and their relationship to group cohomology is far
beyond the scope of this paper. We instead refer the interested reader
to the very comprehensive book~\cite{MR2415834}. In this short section
we will simply outline the facts we need.

\begin{proposition}
  \label{prop:3}
  For $g\geq 2$, the discrete group $\Sp(2g, \setZ)$ has \propT.
\end{proposition}
\begin{proof}
  By Theorem~1.5.3 of \cite{MR2415834}, the locally compact group
  $\Sp(2g, \setR)$ has \propT, and by Theorem~1.7.1, \propT is
  inherited by lattices in locally compact groups. Finally, $\Sp(2g,
  \setZ)$ is known to be a lattice in $\Sp(2g, \setR)$.
\end{proof}

For finitely generated groups, a number of conditions are known to be
equivalent to \propT. The following is quoted from
\cite{MR2415834}, Theorem~3.2.1. 
\begin{theorem}
  \label{thm:3}
  Let $G$ be a locally compact group which is second countable and
  compactly generated. The following conditions are equivalent:
  \begin{enumerate}[\upshape(i)]
  \item $G$ has \propT;
  \item $H^1(G, \pi) = 0$ for every \emph{irreducible} unitary representation $\pi$ of $G$;
  \item $\overline{H^1}(G, \pi) = 0$ for every irreducible unitary
    representation $\pi$ of $G$;
  \item $\overline{H^1}(G, \pi) = 0$ for every unitary representation
    $\pi$ of $G$.
  \end{enumerate}
\end{theorem}
In fact, one can add a fifth element to the list.
\begin{lemma}
  \label{lem:15}
  Let $G$ be a group satisfying the conditions of Theorem~\ref{thm:3}.
  Then conditions {\upshape(i)--(iv)} are also equivalent to
  \begin{enumerate}[\upshape(i)]
    \setcounter{enumi}{4}
  \item $H^1(G, \pi) = 0$ for every unitary representation $\pi$ of $G$.
  \end{enumerate}
\end{lemma}
\begin{proof}
  Clearly (v) implies (ii) and hence the other conditions. By the
  Delorme-Guichardet Theorem (Theorem~2.12.4 in \cite{MR2415834}),
  \propT and \propFH are equivalent for the class of groups
  considered, so \propT implies that $H^1(G, \pi) = 0$ for any
  orthogonal representation $\pi$. Any unitary representation is in
  particular an orthogonal representation, so $H^1(G, \pi) = 0$ for
  any unitary representation as well.
\end{proof}

\begin{corollary}
  \label{cl:3}
  For any unitary representation $\pi\colon \Sp(2g, \setZ)\to \U(V)$,
  the cohomology group $H^1(\Sp(2g, \setZ), V)$ vanishes.
\end{corollary}

\section{Functions on the abelian moduli space}
\label{sec:funct-abel-moduli}

From now on, we let $M = \Hom(\pi_1(\Sigma), \Uone) =
\Hom(H_1(\Sigma), \Uone)$ denote the modulo space of flat $\Uone$
connections on $\Sigma$. The mapping class group acts on $M$ by
$(\phi\cdot \rho)(m) = \rho(\phi\inv m)$ for $\phi\in \Gamma$,
$\rho\in M$ and $m\in H_1(\Sigma)$. This action is smooth and
preserves the measure on $M$, so there are induced actions on
$C^\infty(M)$ and $L^2(M)$ given by $(\phi\cdot f)(\rho) = f(\phi\inv
\rho)$ for smooth or square integrable functions $f$.

Let $\setC$ denote the space of constant functions on
$M$. Then there are splittings of $\Gamma$-modules
\begin{align*}
  C^\infty(M) &\cong C^\infty_0(M) \oplus \setC \\
  L^2(M) &\cong L^2_0(M) \oplus \setC,
\end{align*}
where $C^\infty_0(M)$ and $L^2_0(M)$ denotes the
space of smooth, respectively square integrable, functions with mean
value $0$. The action of $\Gamma$ on $\setC$ is obviously trivial, so
$H^1(\Gamma, \setC) = \Hom(\Gamma, \setC)$, but since the
abelianization of $\Gamma$ is known to be trivial for $g\geq 3$, the
latter is trivial. This yields the isomorphisms
\begin{align*}
  H^1(\Gamma, C^\infty(M)) &\cong H^1(\Gamma,
  C^\infty_0(M)),\\
  H^1(\Gamma, L^2(M)) &\cong H^1(\Gamma, L^2_0(M)).
\end{align*}

\subsection{Pure phase functions}
\label{sec:pure-phase-functions}

Topologically, $M$ is simply a $2g$-dimensional torus. There
is a natural orthonormal basis for $L^2(M)$ parametrized by
$H_1(\Sigma)$, which can be described in several different ways.

The intrinsic definition is rather simple. To a homology element $m\in
H_1(\Sigma)$, we associate the function $\widetilde m$ on
$M$ given by evaluation in $m$, ie. we put
\begin{equation*}
  \widetilde m(\rho) = \rho(m) \in \Uone \subset \setC
\end{equation*}
for $\rho\in M = \Hom(H_1(\Sigma), \Uone)$.

A choice of basis $(x_1, y_1, \ldots, x_g, y_g)$ for $H_1(\Sigma)$
induces a diffeomorphism $M \cong \Uone^{2g}$ given by
\begin{equation*}
  \rho\mapsto (\rho(x_1), \rho(y_1), \ldots, \rho(x_g), \rho(y_g)).
\end{equation*}
Under this identification, the function corresponding to the homology
element $m = a_1x_1 + b_1y_1 + \cdots a_gx_g + b_gy_g$ is simply the
trigonometric monomial
\begin{equation*}
  (z_1, w_1, \ldots, z_g, w_g) \mapsto z_1^{a_1} w_1^{b_1}\cdots z_g^{a_g}w_g^{b_g}
\end{equation*}
on $\Uone^{2g}$. From this description it is clear that the family
$\{\widetilde m \mid m\in H_1(\Sigma)\}$ constitutes an orthonormal
basis for $L^2(M)$.

For any (discrete) set $S$, we use $\ltwo(S)$ to denote the set of
square summable function $S\to\setC$, that is, the set $\{f\colon S\to
\setC \mid \sum_{s\in S} \abs{f(s)}^2 < \infty\}$. We will write such
a function as a formal linear combination $\sum_{s\in S} f_s s$.

\begin{lemma}
  \label{lem:2}
  There is a mapping class group equivariant isomorphism
  \begin{equation}
    \label{eq:6}
    L^2(M) \cong \ltwo(H_1(\Sigma))
  \end{equation}
  where $H_1(\Sigma)$ is considered as a discrete set.
\end{lemma}
\begin{proof}
  We compute
  \begin{equation*}
    (\phi\cdot \widetilde m)(\rho) = \widetilde m(\phi\inv \cdot \rho)
    = (\phi\inv\cdot\rho)(m) = \rho(\phi \cdot m) =
    \widetilde{\phi\cdot m}(\rho),
  \end{equation*}
  proving the equivariance claim.
\end{proof}

Since the element $0\in H_1(\Sigma)$ clearly corresponds to the
constant function $1$ on $M$, we immediately obtain
\begin{lemma}
  \label{lem:16}
  Put $H' = H_1(\Sigma) - \{0\}$, considered as a discrete set. Then
  there is a mapping class group equivariant isomorphism
  \begin{equation}
    \label{eq:7}
    L^2_0(M) \cong \ltwo(H').
  \end{equation}
\end{lemma}
It is very convenient that the action of the mapping class group can
be described by a permutation action on an orthonormal basis.

\subsection{Smooth functions}
\label{sec:smooth-functions}

Now we know that elements of $L^2_0(M)$ can be thought of as
formal linear combinations $\sum_{m\in H'} c_m m$ with $\sum_{m\in H'}
\abs{c_m}^2 < \infty$. We will also need to know under which
conditions a collection of coefficients $(c_m)$ defines a smooth
function. Choose a basis for $H_1(\Sigma)$, and let $\abs{m}$ denote
the norm of a homology element as defined by \eqref{eq:10}.
\begin{proposition}
  \label{prop:2}
  The formal sum $\sum_{m\in H_1(\Sigma)} f_m m$ defines a smooth
  function on $M$ if and only if $\abs{f_m}$ approaches $0$
  faster than any polynomial in $\abs{m}\inv$, or equivalently, if and
  only if for each $k\in\setN$, there is a constant $F_k$ such that
  \begin{equation}
    \label{eq:25}
    \abs{m}^k \abs{f_m} \leq F_k
  \end{equation}
  for all $m\in H_1(\Sigma)$.
\end{proposition}
These conditions are independent of the chosen basis for
$H_1(\Sigma)$.

\section{Cohomology computation}
\label{sec:cohom-comp}

In this final section, we will state and prove the main results of
this paper.

\subsection{Applying Hochschild-Serre}
\label{sec:apply-hochsch-serre}

From now on, we fix a symplectic basis $(x_1, y_1, \ldots, x_g, y_g)$
for $H_1(\Sigma)$, and using this basis we identify $\Sp(H_1(\Sigma))$
with $\Sp(2g, \setZ)$. Consider the short exact sequence
\begin{equation*}
  1 \to \mathcal{T} \to \Gamma \to \Sp(2g, \setZ) \to 1.
\end{equation*}
Since the Torelli group, by definition, acts trivially on
$H_1(\Sigma)$ and hence on $\ltwo(H')$, we are in a position to apply
the exact sequence~\eqref{eq:1}. This now takes the guise
\begin{equation}
  \label{eq:29}
  0\to H^1(\Sp(2g, \setZ), \ltwo(H')) \to H^1(\Gamma, \ltwo(H')) \to
  H^1(\mathcal{T}, \ltwo(H'))^\Gamma.
\end{equation}
\begin{lemma}
  \label{lem:3}
  The last map in \eqref{eq:29} is the zero map.
\end{lemma}
\begin{proof}
  We must prove that any cocycle $u\colon \Gamma\to \ltwo(H')$
  restricts to zero on the Torelli group. To this end, we use the fact
  that the Torelli group is generated by genus~$1$ bounding pair
  maps. Let $t = \tau_\gamma \tau_\delta\inv$ be such a generator for
  $\mathcal{T}$. Since $t$ is invariant under conjugation by
  $\tau_\gamma$, the equivariance \eqref{eq:30} of $u$ restricted to
  $\mathcal{T}$ implies that
  \begin{equation*}
    u(t) = u(\tau_\gamma t\tau_\gamma\inv) = \tau_\gamma u(t)
  \end{equation*}
  which in turn implies that $u(t) = p_\gamma u(t)$. Now, using the
  fact that $\tau_\gamma$ and $\tau_\delta$ acts identically on
  $H_1(\Sigma)$, we know that $p_\gamma = p_\delta$ on
  $\ltwo(H')$. Hence using the fact that $u$ is in fact defined on all
  of $\Gamma$, we obtain
  \begin{equation*}
    u(t) = p_\gamma u(t) = p_\gamma (u(\tau_\gamma) - \tau_\gamma
    \tau_\delta\inv u(\tau_\delta)) = p_\gamma u(\tau_\gamma) -
    \tau_\gamma\tau_\delta\inv p_\delta u(\tau_\delta) = 0
  \end{equation*}
  by Theorem~\ref{thm:4}.
\end{proof}
\begin{corollary}
  \label{cl:2}
  The map
  \begin{equation*}
    H^1(\Sp(2g, \setZ), \ltwo(H')) \to H^1(\Gamma, \ltwo(H'))
  \end{equation*}
  is an isomorphism. \qed
\end{corollary}

Now the first main theorem.
\begin{theorem}
  \label{thm:1}
  The cohomology group
  \begin{equation*}
    H^1(\Gamma, L^2_0(M))
  \end{equation*}
  vanishes.
\end{theorem}
\begin{proof}
  By Corollary~\ref{cl:3}, the cohomology group $H^1(\Sp(2g, \setZ),
  \ltwo(H'))$ vanishes, and by Corollary \ref{cl:2} the same is true
  for $H^1(\Gamma, \ltwo(H'))$. Finally, $\ltwo(H')$ and
  $L^2_0(M)$ are isomorphic as $\Gamma$-modules by
  Lemma~\ref{lem:16}.
\end{proof}

\subsection{Smooth coefficients}
\label{sec:smooth-coefficients}

The second main result looks similar to the first, and its proof is
also based on it.
\begin{theorem}
  \label{thm:2}
  The cohomology group
  \begin{equation*}
    H^1(\Gamma, C^\infty_0(M))
  \end{equation*}
  vanishes.
\end{theorem}
\begin{proof}
  Let $u\colon \Gamma\to C^\infty_0(M)$ by a cocycle.
  Composing with the inclusion $C^\infty_0(M)\to
  L^\infty_0(M) \cong \ltwo(H')$ we may think of $u$ as a
  cocycle $\Gamma\to \ltwo(H')$. Hence, by Theorem~\ref{thm:1} there
  exists an element $f = \sum_{m\in H'} f_m m$ in $\ltwo(H')$ such
  that $u(\gamma) = f-\gamma f$ for each $\gamma\in \Gamma$. We claim
  that $f$ is in fact a smooth function.

  To see this, we must verify the condition \eqref{eq:25} from
  Proposition~\ref{prop:2}. It is clearly enough to do this for all
  large enough $k$, so assume $k\geq 2$. We must find a constant $F_k$
  such that $\abs{m}^k \abs{f_m} \leq F_k$ for all $m\in H'$. Consider
  the $2g$ Dehn twists $\tau_1, \tau_2, \ldots, \tau_{2g}$ in the simple
  closed curves representing our fixed basis for $H_1(\Sigma)$. By
  assumption, for each $j=1,\ldots,2g$, the element
  \begin{equation*}
    u(\tau_j^{\pm1}) = f - \tau_j^{\pm1} f = \sum_{m\in H'} (f_m - f_{\tau_j^{\mp1} m})m   
  \end{equation*}
  defines a smooth function. Putting $g_{m,j}^{\pm} = f_m -
  f_{\tau_j^{\mp1}m}$, there is a constant $G_{k+1}$ such that
  \begin{equation*}
    \abs{m}^{k+1} \abs{g_{m,j}^{\pm}} \leq G_{k+1}
  \end{equation*}
  for all $m\in H'$ and all $j=1,2,\ldots,2g$ (such a constant exist
  for each $\tau_j^{\pm1}$; we may choose the largest of these $4g$
  numbers). We claim that $F_k = G_{k+1}/k$ suffices. To see this,
  observe that $\abs{f_m}\to 0$ as $\abs{m}\to\infty$ since the
  collection $(f_m)$ is square summable. Now, let $m\in H'$ be any
  given element. Choose, by Lemma~\ref{lem:12}, a
  $j\in\{1,2,\ldots,2g\}$ and $\epsilon = \pm 1$ such that
  $\abs{\tau_j^{\epsilon n} m}$ is strictly increasing.  Assume WLOG
  that $\epsilon = +1$. For each $R\geq1$, we have the telescoping sum
  \begin{equation*}
    f_{\tau_j^R m} - f_m = g_{\tau_j^R m, j}^+ + g_{\tau_j^{R-1}m, j}^+ +
    \cdots + g_{\tau_j m, j}^+ = \sum_{r=1}^{R} g_{\tau_j^r m, j}^+
  \end{equation*}
  and hence, since $f_{\tau_j^R m} \to 0$ for $R\to\infty$, we obtain
  \begin{multline*}
    \abs{f_m} = \abs{ \sum_{r=1}^{\infty} g_{\tau_j^r m, j}^+ }
    \leq \sum_{r=1}^{\infty} \abs{g_{\tau_j^r m, j}^+} 
    \leq G_{k+1} \sum_{r=1}^\infty \frac{1}{\abs{\tau_j^r m}^{k+1}}\\
    \leq G_{k+1} \sum_{r=\abs{m}+1}^{\infty} \frac{1}{r^{k+1}}
    < G_{k+1} \int_{\abs{m}}^\infty \frac{1}{r^{k+1}} dr
    = \frac{G_{k+1}}{k \abs{m}^k} 
  \end{multline*}
  using the fact that $\abs{\tau_j^r m}$ is a strictly increasing
  sequence of integers and elementary estimates.

  In case $\epsilon = -1$, we instead use the identity
  \begin{equation*}
    f_m = \sum_{r=1}^{\infty} g_{\tau_j^{-r}m,j}^-
  \end{equation*}
  and proceed exactly as above.
\end{proof}

\bibliographystyle{alphaurl}
\bibliography{../phd}

\end{document}